\documentclass[12pt]{amsart}
\usepackage{amsmath,amsthm,amssymb,amscd,epsfig} 
\usepackage[pagebackref,hypertexnames=false]{hyperref} 

\usepackage{boxedminipage}

\usepackage{mathtools}
\mathtoolsset{showonlyrefs,showmanualtags} 

\allowdisplaybreaks 
\swapnumbers

\newtheorem{theorem}[equation]{Theorem}
\newtheorem{astala}[equation]{Astala's Hausdorff Dimension Distortion Theorem} 

\newtheorem{lemma}[equation]{Lemma}

\newtheorem{proposition}[equation]{Proposition}
\newtheorem{mtheo}[equation]{Main Theorem}

\theoremstyle{remark} 
\newtheorem*{ack}{Acknowledgment}

\theoremstyle{remark}
\newtheorem{remark}[equation]{Remark}

\numberwithin{equation}{section}

\def\R{\mathbb R}
\def\C{\mathbb C}

\def\D{\mathbb D}

\def\H{\mathcal H}

\def\S{\mathcal S}

\def\diam{\operatorname {diam}}

\def\norm#1.#2.{\lVert#1\rVert_{#2}}
\def\Norm#1.#2.{\bigl\lVert#1\bigr\rVert_{#2}}
\def\NOrm#1.#2.{\Bigl\lVert#1\Bigr\rVert_{#2}}
\def\NORm#1.#2.{\biggl\lVert#1\biggr\rVert_{#2}}
\def\NORM#1.#2.{\Biggl\lVert#1\Biggr\rVert_{#2}}

\def\ip#1,#2,{\langle #1,#2\rangle}
\def\Ip#1,#2,{\bigl\langle#1,#2\bigr\rangle}
\def\IP#1,#2,{\Bigl\langle#1,#2\Bigr\rangle}

\def\mid{\,:\,}

\def\XXint#1#2#3{{\setbox0=\hbox{$#1{#2#3}{\int}$}
     \vcenter{\hbox{$#2#3$}}\kern-.5\wd0}}

\def\barint_#1{\mathchoice
  {\mathop{\vrule width 6pt height 3 pt depth -2.5pt
    \kern -8.8pt \intop}\nolimits_{#1}}%
  {\mathop{\vrule width 5pt height 3 pt depth -2.6pt 
    \kern -6.5pt \intop}\nolimits_{#1}}%
  {\mathop{\vrule width 5pt height 3 pt depth -2.6pt
    \kern -6pt \intop}\nolimits_{#1}}%
  {\mathop{\vrule width 5pt height 3 pt depth -2.6pt
    \kern -6pt \intop}\nolimits_{#1}}}

\title[QC Distortion of Hausdorff Measure]
{Astala's Conjecture on Distortion of Hausdorff Measures under Quasiconformal Maps in the Plane}

\author[M.T. Lacey] {Michael T. Lacey$^{(1)}$} 

\address{School of Mathematics \\ Georgia Institute of Technology \\ Atlanta GA 30332 } 

\thanks{$^{(1)}$Research supported in part by a grant from the NSF}

\author[E.T. Sawyer]{Eric T. Sawyer$^{(2)}$} 

\address{ Department  of Mathematics \& Statistics, McMaster University, 1280 Main Street West, Hamilton, Ontario, Canada  L8S 4K1  }
\thanks{$^{(2)}$Research supported in part by a grant from the NSERC}

\author[I. Uriarte-Tuero]{Ignacio Uriarte-Tuero}

\address{ Mathematics Department \\ University of Missouri \\ Columbia, MO 65211 USA \\ \&
Department  of Mathematics \\ Michigan State University \\ East Lansing MI 48824 }


\begin{document}


\begin{abstract}
Let $E \subset \C$ be a compact set, $g: \C \to \C$ be a $K$-quasiconformal map, and  let $ 0<t<2$. Let $\H^t$ denote $t$-dimensional Hausdorff measure. Then 
\begin{equation*}\label{AbsoluteContinuityOfHausdorffMeasuresUnderPushForwardByQCMapsEquationForIntroduction}
\H^t (E) = 0 \quad \Longrightarrow \quad  \H^{t'} (g E) = 0\,, 
\qquad t'=\frac{2Kt}{2+(K-1)t}\,. 
\end{equation*}
This is a refinement of a set of inequalities on the distortion of Hausdorff dimensions  
by quasiconformal maps proved by K.~Astala \cite{astalaareadistortion} and answers in the positive a
conjecture of K.~Astala in \emph{op.~cit.} 
\end{abstract}

\maketitle

\section{Introduction}\label{Introduction}

An orientation-preserving homeomorphism $\phi:\Omega\rightarrow\Omega'$ between planar domains $\Omega,\Omega' \subset \C$ is called {\emph{$K$-quasiconformal}} if  it belongs to  the Sobolev space $W^{1,2}_{loc}(\Omega)$ and satisfies the {\emph{distortion inequality}}
\begin{equation}\label{distortioninequality}
\max_\alpha|\partial_\alpha\phi|\leq K\min_\alpha|\partial_\alpha\phi| \,\,\,\,\,\text{a.e. in }\Omega \;.
\end{equation}
Infinitesimally, quasiconformal mappings carry circles to ellipses with eccentricity at most $ K$.
  Finer properties of quasiconformal mappings can be identified by studying their mapping properties 
with respect to  Hausdorff measure, the primary focus of this paper. 
It has been known since the work of Ahlfors \cite{ahlforsEdition2} that quasiconformal mappings preserve sets of zero Lebesgue measure.
It is also well known that they preserve sets of  Hausdorff dimension zero, since $K$-quasiconformal mappings are H\"older 
continuous with exponent $1/K$, see \cite{mori}. However, they need not  preserve Hausdorff dimension  bigger than zero.
 Gehring and Reich  
 \cite{GehringReichAreaDistortionUnderQCMaps} 
 identified as a conjecture  the precise bounds for the area distortion under quasiconformal mappings, 
 a conjecture verified by the   the groundbreaking work of Astala 
\cite{astalaareadistortion}.
 As a consequence of area distortion, Astala obtained the theorem below, which proved the case $n=2$ of a conjecture of Iwaniec and Martin in $\R^n$ \cite{iwaniecmartinquasiconformalmappingscapacity}.

\begin{astala}
For any compact set $E$ with Hausdorff dimension $0<t<2$ 
and  any $K$-quasiconformal mapping $\phi$ we have
\begin{equation}\label{distortionofdimension}
\frac{1}{K}\left(\frac{1}{t}-\frac{1}{2}\right)\leq\frac{1}{\dim(\phi E)}-\frac{1}{2}\leq K\left(\frac{1}{t}-\frac{1}{2}\right) \ .
\end{equation}
Finally, these bounds are optimal, in that  equality may occur in either estimate. 
\end{astala}

The  question we study concerns 
refinement of the left-hand endpoint above.  
Can it be improved to the level of Hausdorff measures $\H^t$?  Indeed, this is the case.  The next theorem, 
 the main theorem of this paper, answers in the affirmative Astala's Question 4.4 in \cite{astalaareadistortion}.

\begin{mtheo}\label{AbsoluteContinuityOfHausdorffMeasuresUnderPushForwardByQCMapsForIntroduction}
If $\phi$ is a planar $K$-quasiconformal mapping, $0 \le t \le 2$ and $t'=\frac{2Kt}{2+(K-1)t}$,  then we have 
the implication below for all  compact sets $ E\subset \C$. 
\begin{equation}\label{abscont}
\H^t(E)=0 \quad \Longrightarrow \quad \H^{t'}(\phi E)=0,
\end{equation}

\end{mtheo} 
 
Since the inverse of a $K$-quasiconformal mapping is also a $K$-quasiconformal mapping, the following refinement of the right-hand endpoint in \eqref{distortionofdimension} follows: for a compact set $F$, $\H^{t'}(F)>0 \quad \Longrightarrow \quad \H^{t}(\phi F)>0$.

The above theorem  is sharp in two senses. 
Firstly, the hypothesis $\H^t (E) = 0$ cannot be weakened to $\H^t (E) < \infty$ 
while keeping the same conclusion (i.e. the statement ``$\H^t (E) < \infty$ implies $\H^{t'}(\phi E)=0$", under the same conditions of Theorem \ref{AbsoluteContinuityOfHausdorffMeasuresUnderPushForwardByQCMapsForIntroduction}, which has a weaker hypothesis than Theorem \ref{AbsoluteContinuityOfHausdorffMeasuresUnderPushForwardByQCMapsForIntroduction} and hence is a stronger statement than Theorem  \ref{AbsoluteContinuityOfHausdorffMeasuresUnderPushForwardByQCMapsForIntroduction}, is false.)  
Secondly, if we keep the hypothesis $\H^t (E) = 0$, the conclusion $\H^{t'} (\phi E) =0$ cannot be strengthened, to Hausdorff dimension zero 
with respect to a gauge. For any gauge function $h$ satisfying $ \lim_{s \to 0} \frac{s^{t'}}{h(s)} =0$, 
 there exists a compact set $E$ and $K$-quasiconformal mapping $\phi $ with $\H^t (E) = 0$ but 
$\H^h (\phi E) = \infty$. See Theorem 1.7 (a) in \cite{uriartesharpqcstretching} for the relevant examples.


Some instances of this theorem are known, and have connection to significant further properties of 
quasiconformal maps.   
Note that the above classical result of Ahlfors asserts that the theorem is true when $t=2$, while the theorem is obviously true when $t=0$ since $\phi$ is a homeomorphism. 
In fact, for the Lebesgue measure, there is the following precise quantitative bound 
due to \cite{astalaareadistortion} for a properly normalized $K$-quasiconformal mapping $\phi$: 
$$|\phi E|\leq C\,|E|^\frac{1}{K}\,.$$
This bound leads to the sharp Sobolev regularity estimate 
$\phi\in W^{1,p}_{\textup{loc}}(\C)$ for every  $p<\frac{2K}{K-1}$.


A positive answer was also given for the special case $t' =1$ (hence $t = \frac{2}{K+1}$) in 
\cite{astalaclopmateuorobitguriartepreprint}. This special case is  important 
due to its applications towards removability of sets for bounded $K$-quasiregular mappings,
i.\thinspace e.\ a quasiconformal analogue of the celebrated Painlev\'{e}'s problem.
We refer the reader to \cite{tolsasemiadditivityanalyticcapacity} and \cite{astalaclopmateuorobitguriartepreprint} for details. The same paper \cite{astalaclopmateuorobitguriartepreprint}  contains other related results, 
as does Prause \cite{prausedistortionhausdorffmeasuresunderqcmaps}.

Let us give an overview of the proof and the paper.  The highest levels of the argument follow a 
familiar line of reasoning.  Matters are reduced to the case of small dilatation in 
Lemma~\ref{PropositionAbsoluteContinuityOfHausdorffMeasuresUnderPushForwardByQCMapsSmallDilatationV2}.   
Thus, we take a compact set $ E$ with $ t$-Hausdorff measure equal to zero and a $K$-quasiconformal map $ \phi $. 
To provide the conclusion that the $ t'$-Hausdorff measure of $ \phi E$ is zero, we should exhibit 
a covering of $ \phi E$ by (quasi)disks that is arbitrarily small in  $ \mathcal H ^{t'}$-measure.
To do this we should begin with a corresponding covering of $ E$ that is small in the $ \mathcal H ^{t}$-measure. 
The first novelty is that we show that this can be done with certain dyadic cubes (denoted $ P\in \mathcal P$ below) that admit one key additional feature, that they obey a $ t$-packing condition described in Proposition~\ref{PropositionPackingConstruction}.  

Associated with $ \mathcal P$ is a measure $ w _{t, \mathcal P}$, defined in 
\eqref{tefinitionWSubEV2}, which exhibits `$ t$-dimensional' 
behavior, 
reflective of the $ t$-packing condition.
The second novelty is that the Beurling operator, and more generally a standard Calder\'{o}n-Zygmund operator, is bounded on 
$ L ^{2} (w _{t, \mathcal P})$, see Proposition~\ref{PropositionWeightedNormInequalityBeurlingTransformV2}.  This 
fact does \emph{not} follow from standard weighted theory of singular integrals, 
but this new class of measures have enough additional combinatorial structure that a proof of this fact is not difficult to supply. 

The mapping $ \phi $ is then factored into $ \phi = \phi _1 \circ h$ where $ \phi _1$ is the `conformal inside' 
part and $ h$ is the `conformal outside' part.  The conformal inside part admits a relevant estimate that can 
be found in \cite{astalaclopmateuorobitguriartepreprint}, and is recalled below.  The relevant estimate 
on the conformal outside part is new, and uses in an essential way the two novelties just mentioned.  See 
the proof of Lemma~\ref{EstimateForConformalOutside}. It uses Astala's approach for distortion of area \cite{astalaareadistortion}. The conformal inside/outside order of the factorization $ \phi = \phi _1 \circ h$ appears also in \cite{prausedistortionhausdorffmeasuresunderqcmaps}.   

The principal Lemmas are in Section~\ref{PlanAttackAstalasApproach}.  The new lemma on approximating 
Hausdorff content, with control of a packing condition, namely Proposition~\ref{PropositionPackingConstruction}, 
is given in Section~\ref{SectionProofPropositionPackingConstruction}.  Section~\ref{SectionProofPropositionWeightedNormInequalityBeurlingTransformV2} 
contains the proof of  weighted estimate 
for the Beurling operator,  Proposition~\ref{PropositionWeightedNormInequalityBeurlingTransformV2}. 
These two Propositions are combined in
Section~\ref{SectionProofPropositionAbsoluteContinuityOfHausdorffMeasuresUnderPushForwardByQCMapsSmallDilatationV2}.

As usual, in a string of inequalities, the letter $C$ might denote different constants from one inequality to the next.


\begin{ack}
 All authors acknowledge affiliation and support from the Fields Institute in Toronto, during the 
 Thematic Program in Harmonic Analysis.  The first author was  a 
 George Elliot Distinguished Visitor during his stay at the Fields Institute.  
 Part of this work was done the Banff International Research Station, 
 during the workshop 08w5061 Recent Developments in Elliptic and 
 Degenerate Elliptic Partial Differential Equations, Systems and Geometric Measure Theory.  
 The third named author would like to thank Kari Astala, Albert Clop, Guy David, Joan Mateu, Joan Orobitg, Carlos 
 P\'{e}rez, Joan Verdera, and Alexander Volberg for conversations relating to this work.
 
The authors also thank the referees and I. Prause for a detailed reading of the paper and many useful comments that helped clarify the exposition.
\end{ack}

\section{Principal Propositions}\label{PlanAttackAstalasApproach} 

We state the principal Propositions of the paper, with the first being a restatement of the Main Theorem 
for a specific class of quasiconformal mappings, namely those of small dilatation. 

\begin{lemma}\label{PropositionAbsoluteContinuityOfHausdorffMeasuresUnderPushForwardByQCMapsSmallDilatationV2}
Let $ 0<t<2$. Then there is a small constant $ 0<\kappa_0<1 $ ($\kappa_0 = \kappa_0(t)$ is a decreasing function of $t$) so that the following holds.  
Let 
$g: \C \to \C$ be a $K$-quasiconformal map with $ \frac{K-1}{K+1} \leq \kappa_0$. Then we have the 
following implication for all compact subsets $ E\subset \mathbb C $. 
\begin{equation}\label{AbsoluteContinuityOfHausdorffMeasuresUnderPushForwardByQCMapsEquationVersionSmallDilatationV2}
\H^t (E) = 0 \quad \Longrightarrow \quad  \H^{t'} (gE) = 0\,,
\end{equation}
where $t'=\frac{2Kt}{2+(K-1)t}$.
\end{lemma}

\begin{proof}[Proof of Theorem~\ref{AbsoluteContinuityOfHausdorffMeasuresUnderPushForwardByQCMapsForIntroduction}] 
We use  the usual factorization  of a $K$-quasiconformal mapping into those with small dilatation.  
For a fixed $K $-quasiconformal mapping $ g$, we can write 
\begin{equation}\label{FactorizationInto1PlusEpsilonQCMapsV2}
g = g_\lambda  \circ \dots \circ g_2 \circ g_1,
\end{equation}
 so that each $g_i$ is $K_i$-quasiconformal, $K = K_1 \cdots K_\lambda$, and 
$K_i \leq \frac{1+\kappa_0}{1-\kappa_0}$  for all $i = 1,2, \cdots , \lambda$, and $\kappa_0 = \kappa_0(t')$. 
(See \cite{ahlforsEdition2} or \cite{lehto}.)
It follows that the dilatation of each $ g_i$ satisfies $\frac {K_i-1} {K_i+1} \leq \kappa_0(t')$, that is 
Lemma~\ref{PropositionAbsoluteContinuityOfHausdorffMeasuresUnderPushForwardByQCMapsSmallDilatationV2} 
applies to each $ g_i$ individually.

Indeed, let us set $ \tau (t,K)=\frac {2Kt} {2 + (K-1)t}$,  
and inductively define $ \tau _1= \tau (t,K_1)$, 
and $ \tau _{i+1}= \tau (\tau _i,K_{i+1})$.  Let $ E\subset \mathbb C $ be a compact subset of the 
plane with $ \mathcal H ^{t} (E) =0$. It follows from an inductive application of
Lemma~\ref{PropositionAbsoluteContinuityOfHausdorffMeasuresUnderPushForwardByQCMapsSmallDilatationV2}  
(since $ \kappa_0(t') \leq \kappa_0(\tau _i)$ for all $i = 1,2, \cdots , \lambda$)
that we have 
\begin{equation*}
\mathcal H ^{\tau_j} (g_j \circ \cdots \circ g_1 (E))=0 \,, \qquad 1\le j \le \lambda \,.  
\end{equation*}
And it is easily checked that $ \tau_ \lambda = \frac{2Kt}{2+(K-1)t}$, which is the dimension $t'$
in Theorem~\ref{AbsoluteContinuityOfHausdorffMeasuresUnderPushForwardByQCMapsForIntroduction}.

\end{proof}

We state our Proposition on the approximation of Hausdorff content with the $ t$-packing condition. 
Let $ \mathcal P $ be a finite collection of disjoint dyadic cubes in the plane. Let $0<t<2$. We denote the $t$-Carleson packing norm of $ \mathcal P$ as follows: 
\begin{equation}\label{e.packingV2}
\norm \mathcal P . \textup{$t$-pack}. = \sup _{Q} \Biggl[  \ell (Q) ^{-t} \sum _{\substack{P\in \mathcal P\\ P\subset Q }} \ell (P) ^{t} \Biggr] ^{1/t} \,, 
\end{equation}
where the supremum is taken over all dyadic cubes $Q$.  
In this display and throughout this paper, $ \ell (Q)$ denotes the side-length of the cube $ Q$. 
And we say 
that $\mathcal P$ satisfies the $t$-\emph{Carleson Packing Condition} if  $\norm \mathcal P . \textup{$t$-pack}. < \infty$. 

Recall that for a set $E$, $0 \leq s \leq 2$, and $0< \delta \leq \infty$, one defines
\begin{equation}\label{tefinitionHausdorffSDelta}
\H^s_{\delta} (E) = \inf \left\{ \sum_{i=1}^{\infty} \diam(B_i)^s : E \subset \bigcup_{i=1}^{\infty} B_i \,,  \diam(B_i) \leq \delta \; \right\} \,,
\end{equation}
where 
$B_i \subset \C$ is a set, and $\diam(B_i)$ denotes its diameter.
Then one defines the Hausdorff $s$-measure of $E$ to be
\begin{equation}\label{tefinitionHausdorffSMeasure}
\H^s (E) = \lim_{\delta \to 0} \H^s_{\delta} (E) = \sup_{\delta > 0} \H^s_{\delta} (E) \,.
\end{equation}
The quantity $\H^s_{\infty} (E)$ is usually referred to as the Hausdorff content of $E$.

It is well known that in the definition of Hausdorff measure, if instead of covering with balls or arbitrary sets, one covers with dyadic cubes, one obtains an equivalent measure. We will take the dyadic cubes to be closed unless otherwise stated, i.e. of the form $[2^{-k} m_1, 2^{-k} (m_1+1) ] \times [2^{-k} m_2, 2^{-k} (m_2+1)]$, with $k$ a non-negative integer and $m_1, m_2$ integers. Recall also that $\H^t(E) = 0 \Longleftrightarrow \H^t_{\infty} (E) = 0.$ For these and related facts, see e.g. \cite{mattila} or \cite{falconerbook}. 

Only the case $m=2$ of the following Proposition is used below. 

\begin{proposition}\label{PropositionPackingConstruction} Let $m \geq 0$ be an integer. Then there is a positive constant $C$ such that, for any compact $E \subset (0,1)^2 \subset \C$, $0<t<2$, and $\varepsilon > 0$, there is a finite collection of closed dyadic cubes $\mathcal P = \{ P_i \}_{i=1}^{N}$ such that 
\begin{itemize}

\item [(a)] $2^m P_i \cap 2^m P_j = \emptyset$ for $i \neq j$.

\item [(b)] $E \subset \bigcup_{i=1}^{N} 3 \cdot 2^m P_i$.

\item [(c)] $\norm \mathcal P . \textup{$t$-pack}. \leq 1$.

\item [(d)] $\sum_{i=1}^{N} \ell (P_i)^t \leq C \left( \H^t_{\infty} (E) + \varepsilon  \right) $.

\end{itemize}

\end{proposition}

Given $0<t\leq2$ and $\mathcal P $ a collection of pairwise disjoint dyadic cubes, we define the 
 measure $w_{t, \mathcal P}$ associated with $\mathcal P$ by
\begin{equation}\label{tefinitionWSubEV2}
w_{t, \mathcal P} (x) = \sum_j  \ell (P_j)^{t-2} \; \chi_{P_j}(x) ,
\end{equation}
where $\chi_{P_j}$ denotes the characteristic function of $P_j$ and
$\ell (P_j)$ denotes the side-length of $P_j$. Define also
\begin{equation}\label{tefinitionOverlinePV2}
\overline{P} = \bigcup_{i=1}^{N} P_i \,.
\end{equation}
The measure $ w _{t, \mathcal P}$ behaves as does a $t$-dimensional measure, 
namely if $ Q$ is an arbitrary cube (dyadic or not) with sides parallel to the coordinate axes,
then 
\begin{equation}\label{PackingConditionForMeasure}
 w _{t, \mathcal P} (Q)\le 16 \;\norm \mathcal P . \textup{$t$-pack}. ^{t} \; \ell (Q) ^{t} \; .
\end{equation}

We will be concerned with a quasiconformal map $ f$ that is conformal outside of $ \overline P$, 
and we will need an estimate on the diameters of $ f (P_i)$.  $ f$ will have an explicit expression 
as a Neumann series involving the Beurling operator, which we recall here. Let
\begin{equation}\label{tefinitionBeurlingTransformV2}
(\S f)(z) =   -\frac{1}{\pi} \textup{p.v.}\int_{\C} \frac{f(\tau)}{ (z-\tau)^2 } \; dA(\tau),  
\end{equation}
be the Beurling transform. This is an example of a standard singular integral bounded on $L^2(\C)$ (see \cite{steinfatbook}.) 
The second proposition gives a weighted norm inequality with respect to the weight $w_{t, \mathcal P}$ for the
compression of $\S$ to the set $\overline{P}$, i.e. the operator $ \chi_{\overline{P}} \S \chi_{\overline{P}}$, 
assuming that $\mathcal P$ satisfies a Carleson $t$-packing condition.

\begin{proposition}\label{PropositionWeightedNormInequalityBeurlingTransformV2}
Let $0<t<2$ and $\mathcal P = \{ P_i \}_{i=1}^{N}$ be a collection of open dyadic cubes with pairwise disjoint triples, i.e. $3 P_i \cap 3 P_j = \emptyset$ for $i \neq j$. Assume further that $\norm \mathcal P . \textup{$t$-pack}. \leq 1$.
Then there exists an absolute positive constant $C = C(t)$ such that
\begin{equation}\label{e.WeightedNormInequalityBeurlingTransformV2}
\lVert \S ( \chi_{\overline{P}} f ) \rVert_{L^2(w_{t, \mathcal P})} \leq C \lVert  f  \rVert_{L^2(w_{t, \mathcal P})} \,, 
\end{equation}
for all $f \in L^2(\C)$. $C(t)$ is an increasing function of $t$.
\end{proposition}

The proof of this Proposition is presented in
Section~\ref{SectionProofPropositionWeightedNormInequalityBeurlingTransformV2}, 
and follows from elementary bounds on the Beurling operator, and combinatorial properties of the measure 
$ w _{t, \mathcal P}$.   This estimate is new, and does not follow from standard weighted theory. 
  The theory of $ A_2$ weights is built around the assumption that the weights are positive a.e.,
while the weights  $ w _{t, \mathcal P}$ are zero on a large set, and do not admit extensions to 
$ A_2$ weights uniformly in the $ A_2$ characteristic (Cf. Wolff's theorem in \cite[p.439]{garciacuervarubiodefrancia}.)

\section{The Proof of Proposition~\ref{PropositionPackingConstruction}}\label{SectionProofPropositionPackingConstruction}

Given $\varepsilon > 0$, by definition of dyadic Hausdorff content at dimension $t$, there exists a (possibly infinite) collection $\{ Q_n \}$ of closed dyadic cubes such that $E \subseteq \bigcup_{n} Q_n $, and
\begin{equation}\label{CoverEByDyadicCubesApproximatingContentEq1}
\sum_{n} \ell (Q_n)^t \leq \H^t_{\infty} (E) + \varepsilon\,.
\end{equation}

As usual, for $a>0$, denote by $aQ$ the cube concentric to the cube $Q$, but such that $\ell (aQ) = a \cdot \ell (Q)$. By compactness
of $E$, after relabeling indexes, there is a finite number $ N$ for which 
$E \subseteq \bigcup_{n=1}^{N} (3Q_n)^{\circ}$, where $A^{\circ}$
denotes the interior of the set $A$. Since each cube of the form $3Q_n$ is the union of $9$ dyadic cubes of the same
size as $Q_n$, we can 
write, after relabeling, $E \subseteq \bigcup_{n=1}^{N'} Q_n$, where
$Q_n$ are closed dyadic cubes (possibly with overlapping or repeated cubes.)

By selecting the maximal cubes
among the $ Q_n $, and eliminating those $ Q_n $ not intersecting $E$,  we can now assume,  after a relabeling,
that 
\begin{equation}\label{CoverEByDyadicCubesApproximatingContentEq2} \sum_{n=1}^{N} \ell (Q_n)^t \leq 9 \; \left(
\H^t_{\infty} (E) + \varepsilon \right), \end{equation} 
and that the cubes $Q_n$ are dyadic, intersect $E$, and have pairwise
disjoint interiors.

Let $\min \{ \ell (Q_n) \} = 2^{-M}$, 
and call a finite collection of cubes $\mathcal R $ \emph{admissible} denoted by $\mathcal R \in
\operatorname {Adms}$, if  (1) $ \mathcal R$ is a finite collection of dyadic cubes that intersect $E$, thus $ \mathcal R= \{ R_i \}_{i=1}^{H}$ 
for a finite $ H$ and $R_i \cap E \neq \emptyset$ for all $i$; (2)   $2^{-M} \leq \ell (R_i) \leq 1$; $ (3)$  $E \subseteq \bigcup_{i=1}^{H} R_i$; 
and  $ (4)$ they have pairwise disjoint interiors.

We have just seen that $ \operatorname {Adms}$ is non-empty.  The minimum 
$$\min \left\{ \sum_{R_i \in \mathcal R} \ell (R_i)^t :  \mathcal R \in \operatorname {Adms}  \right\} ,$$ 
is achieved, as 
there are only finitely many admissible collections of cubes.
Denote an admissible collection that achieves the minimum as  
$\mathcal T = \{ T_i \}_{i=1}^{M'}$. 
  By \eqref{CoverEByDyadicCubesApproximatingContentEq2}, we have 
\begin{equation}\label{ComparingCollectionOfCubesEq1}
\sum_{i=1}^{M'} \ell (T_i)^t \leq \sum_{j=1}^{N} \ell (Q_j)^t \leq 9 \; \left( \H^t_{\infty} (E) + \varepsilon \right)\,.
\end{equation}

Any minimizer also satisfies a local property: for any dyadic cube $Q$ such that $2^{-M} \leq \ell (Q) \leq 2^0$,
\begin{equation}\label{tPackingConditionForSizesBetween2MAnd1}
\sum_{T_i \subset Q} \ell (T_i)^t \leq \ell (Q)^t.
\end{equation}
Indeed, if $ Q$ intersects $ E$, 
and this inequality did not hold, the cube $Q$ would have been 
selected instead of the cubes $T_i$ with $T_i \subset Q$, contradicting the property of achieving the minimum. 
If the cube $Q$ does not intersect $ E$, then the inequality is trivial.

As an immediate consequence, we get that for any dyadic cube $Q$, irrespective of its size,
\begin{equation}\label{tPackingConditionForAllDyadicCubes} 
\sum_{T_i \subset Q} \ell (T_i)^t \leq \ell (Q)^t \,. 
\end{equation}
 In other words, the cubes $T_i$ satisfy $(c)$ in the statement of Proposition~\ref{PropositionPackingConstruction}.

Thus, $\mathcal T $ satisfies conditions $ (c)$ and $ (d)$ of the conclusion.  To accommodate $ (a)$ and $ (b) $ 
as well, fix an integer  $ m\in \mathbb N \setminus \{ 0 \}$, and fix a cube $T_i \in \mathcal T$. 
Subdivide $T_i$ into its $2^{2m+2}$ dyadic descendants of side-length $2^{-m-1} \ell (T_i)$. 
Let $\widehat{T_i}$ be the dyadic descendant of $T_i$ of side-length $2^{-m-1} \ell (T_i)$ 
whose upper right corner is the center of $T_i$. It is now easy to check that the cubes $\widehat{T_i}$ satisfy 
$(d)$ in the statement of Proposition~\ref{PropositionPackingConstruction} 
(with a larger constant $C$ than the constant obtained for the cubes $T_i$), as well as $(c)$, $(b)$ and $(a)$. Since $t<2$, notice that $C$, which depends on $m$, can be taken independent of $t$.

\section{ Weighted norm inequalities for the Beurling transform}
\label{SectionProofPropositionWeightedNormInequalityBeurlingTransformV2}

We prove the following estimate on the Beurling operator acting on  $ L ^{p} (w _{t, \mathcal P})$ spaces. Note that the same proof applies to any standard Calder\'{o}n-Zygmund singular integral, so we exhibit a whole new class of weights 
with respect to which singular integrals are bounded and yet do not admit extension to $A_p$ weights with uniformly bounded $A_p$ characteristic. For more on non-doubling measures see e.g. \cite[p.439]{garciacuervarubiodefrancia}, \cite{SaksmanLocalMappingHilbertTransform}, \cite{tolsasemiadditivityanalyticcapacity}, \cite{nazarovtreilvolbergTBTheoremNonHomogeneousSpacesActa}, \cite{TolsaWeightedNormInequalitiesCZOperatorsWithoutDoubling}, \cite{OrobitgPerezAPWeightsNonDoublingMeasures}, and the references therein.

\begin{lemma}\label{l.weak-type} Under the assumptions of
Proposition~\ref{PropositionWeightedNormInequalityBeurlingTransformV2}, for any $ 1<p < \infty $, 
and two subsets $ F, G \subset \overline P$, we have the estimate 
\begin{equation} \label{e.weak-type}
{ \int _{G} \lvert  \mathcal S \chi_{F} (x)\rvert \; w _{t, \mathcal P} dx  } 
\le C_{p,t} \lvert  F\rvert _{w _{t, \mathcal P}} ^{1/p} \lvert  G\rvert _{w _{t, \mathcal P}} ^{1-1/p} \,.   
\end{equation}
$C_{p,t}$ is a constant that only depends on $p$ and $t$. For fixed $p$, $C_{p,t}$ is an increasing function of $t$. 
\end{lemma}

Here and throughout, $ \lvert  A\rvert _{w_{t, \mathcal P}}= w_{t, \mathcal P} (A) = \int _{A} w_{t, \mathcal P} dx  $. 
This is the restricted weak-type estimate for $ \mathcal S$ as a bounded operator from the Lorentz space 
$ L ^{p,1} (w _{t, \mathcal P})$ to $ L ^{p, \infty } (w _{t, \mathcal P})$.  A standard interpolation 
then proves Proposition~\ref{PropositionWeightedNormInequalityBeurlingTransformV2} (see e.g. Theorem 3.15 in \cite[p.197]{SteinWeissBook}.)  

To prove this, we split $ \mathcal S$ into a local and non-local part, $ \mathcal S = \mathcal S _{\textup{local}} 
+ \mathcal S _{\textup{non}}$, where  writing the kernel of $ \mathcal S$ as $ K (x,y)$, we define the 
kernel of $ \mathcal S _{\textup{local}}$ to be 
\begin{equation*}
K _{\textup{local}} (x,y)= K (x,y) \sum _{P\in \mathcal P} \chi _{  P} (x)\chi _{  P} (y)
\end{equation*}
On each $ P\in \mathcal P$, $ w _{t, \mathcal P}$ is a constant multiple of Lebesgue measure, hence we can estimate 
the local part directly, using the $ L ^{p } (dx)$-bound for $ \mathcal S$.  
\begin{align*}
\norm \mathcal S _{\textup{local}} f . L ^{p} (w _{t, \mathcal P}). ^{p} 
& = \sum _{P\in \mathcal P} \norm \mathcal \chi _{P} \; S _{\textup{local}} \left( \chi _{P} f \right) . L ^{p} (w _{t, \mathcal P}). ^{p} 
\\
& \le C_p \sum _{P\in \mathcal P} \norm  \chi _{P} f . L ^{p} (w _{t, \mathcal P}). ^{p} 
\\
& \le C_p \norm f. L ^{p } (w _{t, \mathcal P}). ^{p} \,. 
\end{align*}

On the non-local part, we abandon cancellation, and only use the homogeneity of the Beurling kernel.   It is 
also convenient to pass to a combinatorial analog of the non-local operator.  To this end, let us say that 
a collection of (not necessarily dyadic) cubes $ \mathcal Q$ is a \emph{grid} iff for all $ Q, Q'\in \mathcal Q$ we have 
$ Q\cap Q' \in \{\emptyset , Q, Q'\}$.  One can construct a collection of cubes $ \widetilde {\mathcal Q}$ 
so that these conditions hold
\begin{enumerate}
\item  $ \widetilde {\mathcal Q}$ is a union of at most $ 9$ grids. 
\item  For each dyadic cube $ P$ there is a cube $ Q\in \widetilde {\mathcal Q}$ with $ P\subset Q$, 
and $ \lvert  Q\rvert \le C \lvert  P\rvert  $. 
\item For each pair of dyadic cubes $ P , P' $ with $ 3 P\cap 3 P' = \emptyset $, there is a cube $ Q\in \widetilde {\mathcal Q}$ with $ P, P'\subset  Q$ and 
$ \lvert  Q\rvert \le C \operatorname {dist} (P,P') ^2  $. 
\end{enumerate}
Here, $ C$ is an absolute constant.

\begin{proof}
We recall a standard notion used e.g. in \cite[Section 5]{muscalutaothieleMulitilinearOperatorsGivenBySingularMultipliers}. Define a \emph{shifted dyadic mesh} in two dimensions to be 
\begin{equation*}
\widetilde {\mathcal Q} = 
\bigr\{2 ^{j} (k+(0,1) ^2 + (-1) ^{i} \alpha ) \mid i \in \{0,1 \}, j\in \mathbb Z ,\  k \in \mathbb Z ^2 , \ 
\alpha \in \{0, \tfrac 13, \tfrac 23 \} ^2 \bigr\}\,. 
\end{equation*}
Observe that for each cube $ Q\subset \mathbb R ^2 $, there is a $ Q'\in \widetilde {\mathcal Q}$ 
with $ Q\subset \tfrac 9 {10} Q'$ and $\ell (Q') \leq 9 \ell (Q)$.  This is easiest to check in one dimension. 
\end{proof}

Then, it follows that for all functions $ f$ supported on $ \overline P$, and a point $ x\in P$ with $ P\in \mathcal P$,  
\begin{align*}
\lvert  \mathcal S  _{\textup{non}}f  (x)\rvert
& \le  \sum _{\substack{P'\in \mathcal P\\ P'\neq P }} 
\int _{P'} \lvert   K (x,y) f (y)\rvert \; dy 
\\
&   \le C  \sum _{\substack{P'\in \mathcal P\\ P'\neq P }} 
\int _{P'} \lvert  f (y)\rvert \; \frac {dy} { \operatorname {dist} (P,P') ^2}  
\\
& \le C \operatorname S _{\widetilde {\mathcal Q}}   \lvert  f\rvert  (x)
\end{align*}
where we define for any collection of cubes $\mathcal Q $, 
\begin{equation}\label{e.SQ}
\operatorname S _{\mathcal Q} f (x) = \sum _{ \substack{ Q\in \mathcal Q \\ Q \text{  non-local} } } \frac { \chi _Q (x)} { \ell (Q) ^2 } \int _{Q} 
f (y)\;  dy \,. 
\end{equation}

Here we say that an arbitrary cube $Q$ (dyadic or not) with sides parallel to the coordinate axes is \emph{non-local} if there exist $P_1, P_2 \in \mathcal P$ such that $P_i \cap Q \neq \emptyset$ for $i=1,2$. It follows (since $3P_1 \cap 3P_2 = \emptyset$) that if $Q$ is non-local, then $\ell(P) \leq \ell(Q)$ if $P \in \mathcal P$ and $P \cap Q \neq \emptyset$.

Given the collection of cubes $\mathcal P$ (which is fixed throughout this section), and given a grid $\mathcal Q$, there is a unique subcollection of cubes $\mathcal Q' = \{ \text{ non-local cubes of the underlying grid } \mathcal Q \; \}$.

Thus, for the proof of Lemma~\ref{l.weak-type}, it suffices to consider only collections of cubes $\mathcal Q'$ (and we restrict our attention to such collections of cubes for the rest of this section) and to prove 

\begin{lemma}\label{l.Sweak} Under the assumptions of Lemma~\ref{l.weak-type}, 
for 
the collection of non-local cubes $\mathcal Q'$ associated to any grid $\mathcal Q$ 
we have the inequality
\begin{equation}\label{e.RestrictedWeakTypeEstimateForSQPrime}
\int _G  \left[ S _{\mathcal Q'} \chi _F  \right] \; w _{t, \mathcal P} dx 
\le C_{p,t} \lvert  F\rvert _{w _{t, \mathcal P}} ^{1/p} 
\lvert  G\rvert _{w _{t, \mathcal P}}  ^{1-1/p} \,, 
\qquad 1<p< \infty \,. 
\end{equation}
For fixed $p$, $C_{p,t}$ is an increasing function of $t$. 
\end{lemma}

We turn to the proof.   There are two points to observe. 
Consider the $ w_{t,\mathcal P}$-maximal function defined by 
\begin{equation}\label{e.Mw}
\operatorname M_{t} g = \sup _{Q \in \mathcal Q'} \frac {\chi _{Q}} {\lvert Q \rvert_{w_{t,\mathcal P}}} \int _Q g (y) 
w _{t, \mathcal P} (y) \; dy 
\end{equation}
This operator maps $ L ^{1} (w_{t,\mathcal P})$ to $ L ^{1, \infty } (w_{t,\mathcal P})$, that is, 
\begin{equation}\label{e.Mw1}
\lambda \lvert \{ \operatorname M _t g > \lambda \} \rvert_{w_{t,\mathcal P}} \le \norm g. L ^{1} (w_{t,\mathcal P}).  \,, \qquad 0<\lambda <\infty \,. 
\end{equation}
Indeed, this is a maximal inequality  true for all weights, and follows immediately from the 
usual Covering Lemma proof, which is quite simple in this  context, as $ \mathcal Q$ is a grid. 

For $ F, G \subset \overline P $, if $ 8 \lvert  F\rvert _{w _{t, \mathcal P}}  
\le \lvert  G\rvert _{w _{t, \mathcal P}} $, we take $ F'=F$.  Otherwise we define 
\begin{equation}\label{e.G''}
F'= F \cap \bigl\{ \operatorname M_t \chi_{G} \leq 2 w_{t, \mathcal P}(G)/w_{t, \mathcal P} (F) \bigr\}\,. 
\end{equation}
By the weak-$L ^{1} (w _{t, \mathcal P}) $ inequality for $ \operatorname M _{t}$ we 
see that $ \lvert  F' \rvert _{w _{t, \mathcal P}} \ge \tfrac 12 \lvert  F\rvert _{w _{t, \mathcal P}}  $. 
(In the argot of \cite[Section 3]{muscalutaothieleMulitilinearOperatorsGivenBySingularMultipliers}, $ F'$ is a major subset of $ F$.) 
  We show that 
\begin{equation}\label{e.''}
\int _{G} \left[ \operatorname S _{\mathcal Q'} \chi_{F'} \right] \; w _{t, \mathcal P} dx 
\le C_t
 \min \{ w_{t,\mathcal P}(F) \,,\, w_{t, \mathcal P}(G)\}\,. 
\end{equation}
Upon iteration of inequality \eqref{e.''}, we see that we actually have inequality \eqref{e.''}, with $ F'=F$ 
on the left hand side of the inequality, and $C_t$ replaced by $C_t \log \left( 2+ \frac{\lvert  F\rvert _{w _{t, \mathcal P}}}{\lvert  G\rvert _{w _{t, \mathcal P}}} \right)$. Indeed, with $F = F_0$ and $F' = F_1$ we now apply \eqref{e.''} with $F_0$ replaced by $F_0 \setminus F_1$, and $F_2$ the corresponding major subset of $F_0 \setminus F_1$. We continue the iteration until $ 8 \lvert  F_n \rvert _{w _{t, \mathcal P}}  
\le \lvert  G\rvert _{w _{t, \mathcal P}} $, which occurs with $n \lesssim \log \left( 2+ \frac{\lvert  F\rvert _{w _{t, \mathcal P}}}{\lvert  G\rvert _{w _{t, \mathcal P}}} \right)$. From this inequality we immediately obtain \eqref{e.RestrictedWeakTypeEstimateForSQPrime}:
\begin{equation}\label{e.'''}
\int _{G} \left[ \operatorname S _{\mathcal Q'} \chi_{F} \right] \; w _{t, \mathcal P} dx 
\le C_t \log \left( 2+ \frac{\lvert  F\rvert _{w _{t, \mathcal P}}}{\lvert  G\rvert _{w _{t, \mathcal P}}} \right)
 \min \{ w_{t,\mathcal P}(F) \,,\, w_{t, \mathcal P}(G)\}\, \leq C_{p,t} \lvert  F\rvert _{w _{t, \mathcal P}} ^{1/p} 
\lvert  G\rvert _{w _{t, \mathcal P}}  ^{1-1/p} \,, 
\end{equation}
for $ 1<p< \infty \,$, which reduces the proof of Lemma~\ref{l.Sweak} to showing \eqref{e.''}.

We now turn to the proof of \eqref{e.''}. 
\begin{align}
\int _{G} \left[ \operatorname S _{\mathcal Q'} \chi_{F'} \right] \; w _{t, \mathcal P} \; dx & = 
\sum _{Q\in \mathcal Q' } 
\frac { \lvert   F' \cap Q \rvert} { \ell (Q) ^2 }  \lvert G \cap Q\rvert _ {w _{t, \mathcal P}} =
%
\sum _{Q\in \mathcal Q' } 
\frac { \lvert   F' \cap Q \rvert} { \ell (Q) ^{2-t} } \; \frac{ \lvert G \cap Q\rvert _ {w _{t, \mathcal P}} }{ \ell (Q) ^{t} }  \nonumber \\
& \leq  \label{e.SSSS2}
\min \left\{ 16 \;\norm \mathcal P . \textup{$t$-pack}. ^{t} ,  32 \frac{ w_{t, \mathcal P}(G)}{w_{t, \mathcal P} (F)} \right\} \; \sum _{Q\in \mathcal Q' } \frac { \lvert   F' \cap Q \rvert} { \ell (Q) ^{2-t} } := A \; \sum _{Q\in \mathcal Q' } \frac { \lvert   F' \cap Q \rvert} { \ell (Q) ^{2-t} }  \\
& = A \; \sum _{Q\in \mathcal Q' } \; \sum_{P : P \cap Q \neq \emptyset} \frac{\lvert   F' \cap P \cap Q \rvert}{\ell (Q) ^{2-t}} \leq A \; \sum _{P \in \mathcal P } \sum _{ \substack{ Q\in \mathcal Q' \\ Q \cap P \neq \emptyset} } \lvert   F' \cap P \rvert \; \frac{1}{\ell (Q)^{2-t}}  \\
& \leq \label{e.SSSS3}
A \; C_t \; \sum _{P \in \mathcal P } \frac{\lvert   F' \cap P \rvert}{\ell (P)^{2-t}} = 
A \; C_t \; \lvert F' \rvert_{w_{t, \mathcal P}}  \\
& \leq C_t' \; \min \{ w_{t,\mathcal P}(F') \,,\, w_{t, \mathcal P}(G)\}\, \leq C_t' \; \min \{ w_{t,\mathcal P}(F) \,,\, w_{t, \mathcal P}(G)\}\,. 
\end{align}

In passing to \eqref{e.SSSS2}, we have used the packing condition (see \eqref{PackingConditionForMeasure}) and the definition of 
$ F'$  in \eqref{e.G''}, to wit if $ \lvert  Q\cap F'\rvert \neq 0  $, then necessarily 
\begin{equation*}
 \frac {\lvert  G\cap Q\rvert _{w_{t, \mathcal P}} } 
 { \ell   (Q) ^{t}  }
 \le 16
  \frac {\lvert  G\cap Q\rvert _{w_{t, \mathcal P}} } 
 { \lvert  Q \rvert _{w_{t,\mathcal P}}   } 
\le 32 \frac {\lvert G\rvert_{w_{t, \mathcal P}} } {\lvert F\rvert_{w_{t, \mathcal P}}} \,.  
\end{equation*}
In passing to \eqref{e.SSSS3}, we have used that for any fixed scale $2^{-\ell}$, there are at most $4$ cubes $Q \in \mathcal Q'$ such that $Q \cap P \neq \emptyset$ and $\ell(Q) = 2^{-\ell}$, and also that any such $Q$ satisfies $\ell(Q) \geq \ell(P)$, since $Q$ is non-local. Note that $C_t$ and $C_t'$ are increasing functions of $t$.

\section{The Proof of Lemma~\ref{PropositionAbsoluteContinuityOfHausdorffMeasuresUnderPushForwardByQCMapsSmallDilatationV2}}
\label{SectionProofPropositionAbsoluteContinuityOfHausdorffMeasuresUnderPushForwardByQCMapsSmallDilatationV2}

We use a familiar scheme, which we recall here.   We have already seen how to approximate the 
$ t$-Hausdorff content of a set $E $ by a finite union of cubes.  We can therefore assume that 
$ E$ is in fact a finite union of cubes, and we approximate the Hausdorff content of the image of 
$ E$. 
Applying Stoilow factorization methods, a normalized version of the mapping 
$\phi$ is written as $\phi= \phi_1 \circ h$, where both $h,\phi_1:\C\rightarrow\C$ are principal $K$-quasiconformal mappings, such that $h$ is conformal in the complement of the set $E$ and $\phi_1$ is conformal on the set $F=h(E)$.
One then studies the mapping properties of the two functions $ \phi _1 $ and $ h$ separately, referred to 
the `conformal inside' and the `conformal outside' parts, respectively. Recall that a principal $K$-quasiconformal mapping is a $K$-quasiconformal mapping that is conformal outside 
$\overline{\D}$ and is normalized by $\phi(z)-z=O\left(\frac{1}{|z|}\right)$ as $|z|\to\infty$.

The conformal inside part has already been addressed, in 
\cite{astalaclopmateuorobitguriartepreprint}, and we recall the relevant result in 
 Theorem~\ref{QCConformalInside} below.   The conformal outside part is new, and the point we turn to now.

The following lemma is often used in the theory of extrapolation of $A_p$ weights, 
and we use it in a similar way to the way it is used in that theory.

\begin{lemma}\label{HoelderInequalityForPLessThan1}
Let $f,g \geq 0$ be measurable functions. Then, if $0<p<1$,
\begin{equation}\label{HoelderInequalityForPLessThan1Inequality}
\int fg \geq\lVert  f  \rVert_{p}\lVert  g  \rVert_{p'}, 
\end{equation}
where $\frac{1}{p} + \frac{1}{p'} =1$ (hence $p' <0$), 
$\lVert  f  \rVert_{p} = \left( \int |f|^p \right)^{\frac{1}{p}}$, and
\begin{equation*}
\lVert  g  \rVert_{p'} = \left( \int |g|^{p'} \right)^{\frac{1}{p'}} = 
\frac{1}{ \left( \int \frac{1}{|g|^{-p'}} \right)^{\frac{1}{-p'}} } \,. 
\end{equation*}
As a consequence, 
\begin{equation}\label{HoelderInequalityForPLessThan1Equality}
\lVert  f  \rVert_{p} = \inf_{ \substack{ g \;:\;  \lVert  g  \rVert_{p'} =1 }  } \int fg.
\end{equation}

\end{lemma}

\begin{proof} The inequality \eqref{HoelderInequalityForPLessThan1Inequality} follows easily from the usual H\"{o}lder's
inequality (i.e. with $p>1$.) The case of equality in \eqref{HoelderInequalityForPLessThan1Equality} is attained by
taking $g = \frac{f^{p-1}}{ \lVert f \rVert_{p}^{p-1} }$.

\end{proof}

We will use the following notation. For a finite collection of pairwise disjoint dyadic cubes $\mathcal P = \{ P_j  \}_{j=1}^{N}$, let
\begin{equation}\label{tefinitionBetaJ}
\beta_j = \frac{ \left[ \ell(P_j)^2 \right]^{( \frac{t}{2} -1 )} }{  \left\{ \sum_{k=1}^{N} \left[ \ell(P_k)^2 \right]^{\frac{t}{2}}   \right\}^{( \frac{t}{2} -1 ) \frac{2}{t}  }  }.
\end{equation}
(Compare with $g$ in the proof of Lemma \ref{HoelderInequalityForPLessThan1}.) Also, let $E = \overline{P} = \bigcup_j P_j$, let 
\begin{equation}\label{tefinitionTildeWSubE}
\widetilde w _{t, \mathcal P} (x) = \sum_{j} \beta_j \cdot \chi_{P_j}(x), 
\end{equation}
which is a constant multiple of $ w _{t,\mathcal P}$, as defined in \eqref{tefinitionWSubEV2}.

The conformal outside Lemma states that the quasi-conformal image of $ \mathcal P$ 
has controlled distortion, in the $\ell^{t}$-quasinorm.

\begin{lemma}\label{EstimateForConformalOutside} Let $ 0<t<2$.  There is a positive constant $ \varepsilon _0$ (which is a decreasing function of $t$) so that the following holds. 

Let $\mathcal P = \{ P_j \}_{j=1}^{N}$ be a finite collection of dyadic cubes which satisfy the $t$-Carleson packing condition 
$\norm \mathcal P . \textup{$t$-pack}. \leq C$. 
Assume further that the cubes $3 \; P_j$ are pairwise disjoint.

Let $E = \overline{P} = \bigcup_{j} P_j$ and let $f: \C \to \C$
be a principal $K$-quasiconformal mapping which is conformal outside the compact set $E$, 
with $ \frac{K-1}{K+1} < \varepsilon _0$. 

Then, there is a constant $C(K,t)$ which depends only on  $K$ and $t$ (which, for fixed $K$, is an increasing function of $t$) such that
\begin{equation}\label{MorallyEstimateForDHausdorffContent}
\sum_{j=1}^{N} \diam(f(P_j))^t \leq C(K,t) \sum_{j=1}^{N} \ell(P_j)^t.
\end{equation}

\end{lemma}

Prause \cite{prausedistortionhausdorffmeasuresunderqcmaps} proved results somewhat in the spirit of
Lemma~\ref{EstimateForConformalOutside} below, but for different Hausdorff measures, which give a weaker conclusion than
the statement \eqref{abscont}.
Our Lemma, and in particular the hypothesis on $ t$-packing,  is informed by the 
counterexample of  Bishop \cite{bishopdistortionofdisksbyconformalmaps}.

\begin{proof}


By Lemma~\ref{HoelderInequalityForPLessThan1},
with $\beta_j$ as in \eqref{tefinitionBetaJ}, and $\widetilde w _{t, \mathcal P} (x)$, as in \eqref{tefinitionTildeWSubE} 
, by quasi-symmetry  we get
\begin{align} \label{xx}
\left(  \sum_{j=1}^{N} \diam(f(P_j))^t \right)^{\frac{2}{t}} 
& =  \inf_{ \substack{ \alpha_j >0  
\\  1= \lVert \left\{  \alpha_j  \right\} \rVert_{   \ell ^{(t/2)'} }  } }  
\left\{  \sum_{j=1}^{N}  \diam(f(P_j))^2 \;  \alpha_j  \right\}  \\
& \leq   \sum_{j=1}^{N}  \diam(f(P_j))^2 \;  \beta_j 
\\  
&\leq   
C(K) 
\ \int_{E} J(z,f) \; \widetilde w _{t, \mathcal P}(z) \; dA(z)  
\end{align}
Here $J(z,f)$ denotes the Jacobian (determinant) of $f$ at $z$. 

We  follow Astala's approach for his area distortion theorem \cite[p.50]{astalaareadistortion} (see
also \cite{astalaiwaniecmartin}), equipped with the new results of this paper.
The central role of the Beurling operator is indicated by the identity 
\begin{equation}\label{FZAndFZBarRelatedByBeurling} 
f_{z} = 1 + \S(f_{\overline{z}}). 
\end{equation}
Using the trivial inequality $|2 \operatorname{Re}(a)| \leq 2|a| \leq |a|^2 +1$, and that $J(z,f) = |f_{z}|^2 - |f_{\overline{z}}|^2$ (see e.g. (9) in \cite[p.6]{ahlforsEdition2}, or \cite{astalaiwaniecmartin}), we can estimate 
\begin{align} \notag
\int_{E} J(z,f) \; \widetilde w _{t, \mathcal P} (z) \; dA(z) & = 
\int_{E} ( |f_{z}|^2 - |f_{\overline{z}}|^2 ) \; \widetilde w _{t, \mathcal P} (z) \; dA(z)  \\
& =   \int_{E} ( 1 + 2 \operatorname{Re}(\S(f_{\overline{z}})) +  |\S(f_{\overline{z}})|^2  - 
|f_{\overline{z}}|^2  ) \; 
\widetilde w _{t, \mathcal P} (z) \; dA(z) 
\\ \notag
& \leq  2 \int_{E} ( 1 +  |\S(f_{\overline{z}})|^2  
) \; \widetilde w _{t, \mathcal P} (z) \; dA(z)  
\\ \notag
& =  2 \Big\{  \int_{E}  \widetilde w _{t, \mathcal P} (z) \; dA(z) +
 \int_{E}  | \S(f_{\overline{z}}) |^{2}  \; \widetilde w _{t, \mathcal P} (z) \; dA(z)   \Big\}
 \\
\label{EstimateIntegralOfJacobianUsingBeurlingTransformIdentityAndSplittingInto3Parts}
&= 2 \left\{ \operatorname I_{1} + \operatorname I_{2} 
\right\} \,.
\end{align}

Notice that
$
\operatorname I_{1} = \sum_{j=1}^{N}  \ell (P_j)^2 \;  \beta_j 
$.  We shall bound the other term by a multiple of $ I_1$. Indeed, 
with respect to $\operatorname I_{2}$, since $\widetilde w _{t, \mathcal P}$ and $w_{t, \mathcal P}$ only differ by a multiplicative constant 
the Beurling operator has the same operator norm on $ L ^2 (\widetilde w _{t, \mathcal P})$ and 
$ L ^2 ( w _{t, \mathcal P})$. 
And so by Proposition~\ref{PropositionWeightedNormInequalityBeurlingTransformV2}, 
\begin{equation}\label{EstimateForTermI2InTermsOfI3}
\operatorname I_{2} = \int_E  |\S(f_{\overline{z}})|^{2} \; \widetilde w _{t, \mathcal P} (z) \; dA(z) \leq C(t) \int_{E}  |f_{\overline{z}}|^{2}  \; \widetilde w _{t, \mathcal P} (z) \; dA(z) =: C(t) \; \cdot \operatorname I_{3}
\end{equation}

Turning to $ I_3$, the Beurling operator is again decisive.  
Recall  the representation of $f_{\overline{z}}$ as a power series in the Beltrami coefficient $\mu$. Namely,
\begin{equation}\label{PowerSeriesOfFZBarInTermsOfMu}
f_{\overline{z}} = \mu f_{z} = \mu + \mu \S (\mu) + \mu \S (  \mu \S ( \mu ) ) + \cdots
\end{equation}
 This  is obtained upon multiplying \eqref{FZAndFZBarRelatedByBeurling} by $\mu$, writing 
$
f_{\overline{z}} = \left( \operatorname{Id}  - \mu \S \right)^{-1} (\mu)
$ 
and using the standard Neumann series 
\begin{equation}\label{ExpressionInverseIdMinusMuSAsNeumannSeries}
\left( \operatorname{Id}  - \mu \S \right)^{-1} = \operatorname{Id} + \mu \S + \mu \S \mu \S + \mu \S \mu \S \mu \S +
\cdots \,. 
\end{equation}
As we shall see, this series converges in $ L ^{2} ( w _{t,\mathcal P} )$ for small (depending on $t$) $\parallel \mu \parallel_{\infty}$ by Proposition \ref{PropositionWeightedNormInequalityBeurlingTransformV2}.

Observe the two inequalities 
\begin{align}\label{EstimateFirstTermPowerSeriesOfFZBarInTermsOfMu}
\left(  \int_{E} |\mu|^{2}  \; \widetilde w _{t, \mathcal P} (z) \; dA(z)  \right)^{\frac{1}{2}} &\leq \lVert \mu  \rVert_{\infty} \left(  \int_{E} \chi_{E} \cdot\widetilde w _{t, \mathcal P} (z) \; dA(z)  \right)^{\frac{1}{2}} =  \lVert\mu  \rVert_{\infty} \left( \operatorname I_{1} \right)^{\frac{1}{2}} \,,
\\ 
\label{EstimateKThTermPowerSeriesOfFZBarInTermsOfMu}
\left(  \int_{E} |\mu \S (g)|^{2}  \; \widetilde w _{t, \mathcal P} (z) \; dA(z)  \right)^{\frac{1}{2}} 
&\leq 
\lVert \mu  \rVert_{\infty}  \cdot \lVert\S  \rVert_{L^{2}( \widetilde w _{t, \mathcal P} )} 
\left(  \int_{E} |g|^{2}  \; \widetilde w _{t, \mathcal P} (z) \; dA(z)  \right)^{\frac{1}{2}} \,. 
\end{align}
The second inequality  is applied to the sequence of functions 
 $g=\mu$, $g= \mu \S (\mu)$, $g = \mu \S (\mu \S (\mu))$ and so on.
 Using the triangle inequality in \eqref{PowerSeriesOfFZBarInTermsOfMu} in the 
 $L^{2}( \widetilde w _{t, \mathcal P} )$ norm gives 
\begin{equation}\label{EstimateWholePowerSeriesOfFZBarInTermsOfMu}
\left( \operatorname I_{3} \right)^{\frac{1}{2}} \leq  \Vert  \mu  \Vert _{\infty} 
\Bigl\{  \sum_{n=1}^{\infty}  \left[ \Vert  \mu  \Vert _{\infty}   
\Vert  \S  \Vert _{L^{2}( \widetilde w _{t, \mathcal P} )}  \right]^{n}  \Bigr\}  \left( \operatorname I_{1} \right)^{\frac{1}{2}}.
\end{equation}

The middle term on the right is bounded if we demand 
\begin{equation}\label{SecondEstimateOnKappaInTermsOfP0}
\Vert  \mu  \Vert _{\infty} < \varepsilon _0= \bigl[
{2 \; \Vert  \S  \Vert _{L^{2}( \widetilde w _{t, \mathcal P} )  \to L^{2}( \widetilde w _{t, \mathcal P} )  } } 
\bigr] ^{-1} < 1 \,. 
\end{equation}
This is the $ \varepsilon _0$ required in the statement of Lemma \ref{EstimateForConformalOutside} (and hence $ \varepsilon _0$ is a decreasing function of $t$.)  It follows that 
\begin{equation}\label{EstimateOfI3InTermsOfI1}
\operatorname I_{3} \leq \operatorname I_{1} \,.
\end{equation}

From \eqref{xx}, \eqref{EstimateIntegralOfJacobianUsingBeurlingTransformIdentityAndSplittingInto3Parts}, 
\eqref{EstimateForTermI2InTermsOfI3}, and \eqref{EstimateOfI3InTermsOfI1}, it follows that 
\begin{equation}\label{EstimateIntegralOfJacobianUsingBeurlingTransformIdentityAndSplittingInto3PartsII}
\left(  \sum_{j=1}^{N} \diam(f(P_j))^t \right)^{\frac{2}{t}} \leq C(K) 
\int_{E} J(z,f) \; \widetilde w _{t, \mathcal P} (z) \; dA(z) \leq C' (K,t)\operatorname I_{1} \,.
\end{equation}
It remains to bound $ \operatorname I_1$  by the right hand side of \eqref{MorallyEstimateForDHausdorffContent}. 

But it follows by construction (recall the parenthetical comment right after \eqref{tefinitionBetaJ}) that 
\begin{equation}\label{IdentityForTermI1PartII}
\operatorname I_{1} = \sum_{j=1}^{N}  \ell (P_j)^2 \;  \beta_j= \Vert  \left\{  \ell (P_j)^2  \right\}_{j=1}^{N}   \Vert _{\ell^{\frac{t}{2}} } =
 \left(  \sum_{j=1}^{N} \ell (P_j)^t \right)^{\frac{2}{t}}.
\end{equation}
This completes the proof. 
 
\end{proof}

Recall that it is known how to deal with the quasiconformal map which is `conformal inside'. Namely, recall the following
\begin{theorem}\label{QCConformalInside}
Let $\phi: \C \to \C$ be a principal $K$-quasiconformal mapping which is conformal outside $\D$.
Let $\{ S_j \}_{j=1}^{N}$ be a finite family of pairwise disjoint quasi-disks in $\D$, such that $S_j = f(D_j)$ for a single $K$-quasiconformal map $f$ and for disks (or cubes) $D_j$, and assume that
$\phi$ is conformal in $\Omega = \bigcup_j S_j$. Then for any $t \in (0,2]$ and $t'=\frac{2Kt}{2+(K-1)t}$, we have 
\begin{equation}\label{QCConformalInsideInequality}
\left(\sum_{j=1}^N\diam(\phi(S_j))^{t'}\right)^\frac{1}{t'}\leq C(K)\,
\left(\sum_{j=1}^N\diam(S_j)^t\right)^{\frac{1}{tK}}
\,. 
\end{equation}
\end{theorem}

Theorem~\ref{QCConformalInside} can be found in \cite[(2.6)]{astalaclopmateuorobitguriartepreprint} stated for disks $D_j$, but the proof works for $K$-quasi-disks (more precisely, we use it for ``$K$-quasi-squares", i.e. the image under a single $K$-quasiconformal map - where $K$ will be typically close to $1$ - of squares.) It should be emphasized that for a general quasiconformal mapping $\phi$ we have $J(z,\phi)\in L^p_{loc}$ only for $p<\frac{K}{K-1}$. The improved  integrability $p=\frac{K}{K-1}$  under the
extra assumption that $\phi \vert_{ \Omega}$ is conformal was shown in \cite[Lemma 5.2]{astalanesi}. This phenomenon is
crucial for the proof of Theorem~\ref{QCConformalInside}, since we are studying Hausdorff measures rather than
dimension. Note that Theorem~\ref{QCConformalInside} is also implicit in \cite{astalaareadistortion} (see Corollary 2.3
and the variational principle in p.48.)

At this point we prove Astala's conjecture for the case of small dilatation, 
Lemma~\ref{PropositionAbsoluteContinuityOfHausdorffMeasuresUnderPushForwardByQCMapsSmallDilatationV2}.

\begin{proof}[Proof of
Lemma~\ref{PropositionAbsoluteContinuityOfHausdorffMeasuresUnderPushForwardByQCMapsSmallDilatationV2}.] 

We first give the argument that allows us to reduce to the usual normalizations. It is a standard argument, but we give
it for convenience.

Let $\tau$ be a M\"obius transformation fixing $\infty$. The dilatation $K$ of $g$, let us call it $K_g$, is the same as that of $g
\circ \tau$, i.e. $K_{g} = K_{g \circ \tau}$. Also, $\H^t (E) = 0 \Longleftrightarrow \H^t (\tau (E)) = 0$. Consequently,
without loss of generality, we can assume that $E \subset (\frac{1}{32}, \frac{1}{16})^2 \subset \frac{1}{8} \D$.

Let $\mu_g$ be the Beltrami coefficient for $g$. Let $\varphi$ be the (unique) principal homeomorphic solution to the Beltrami equation 
$$\overline{\partial} \varphi = \left( \chi_{\D} \mu_g \right) \partial \varphi \,.$$
Then, by Stoilow's factorization, we have that $g = \psi \circ \varphi$, where $K_g = K_{\psi} = K_{\varphi}$, both $\psi$ and $\varphi$ are $K$-quasiconformal maps, $\varphi$ is principal and $\psi$ is conformal in $\varphi (\D)$. 

Since $\psi$ is conformal in a neighbourhood of $\varphi (E)$, by Koebe's distortion theorem (see e.g. \cite{pommerenke}), $0 < c_{\psi} \leq \inf_{\varphi (E)} |\psi'(z)|  \leq \sup_{\varphi (E)} |\psi'(z)| \leq C_{\psi} < \infty$, and hence $\psi$ is bi-Lipschitz in $\varphi (E)$. Therefore $\H^{t'} (S) = 0 \Longleftrightarrow \H^{t'} (\psi (S)) = 0$ if $S \subset \varphi ( [\frac{1}{32}, \frac{1}{16}]^2 )$.
Consequently, without loss of generality, we can further assume that $g$ is a principal mapping.

Consider $\varepsilon >0$ and use Proposition~\ref{PropositionPackingConstruction}, with $ m=2$, to obtain a 
collection of cubes $\mathcal P = \{P_i\}$ satisfying the conclusions of Proposition~\ref{PropositionPackingConstruction} with
respect to the compact set $E$. Denote $\Omega = \left(  \bigcup_i P_i \right)^{\circ}$.

Following \cite{astalaareadistortion}, decompose $g = \phi \circ f $, where both $\phi$ and $f$ are principal
$K$-quasiconformal mappings, $f$ is conformal outside $\overline{\Omega}$, 
and $\phi$ is conformal in $f(\Omega) \cup (\C \setminus \D)$. 
Recall that Lemma~\ref{EstimateForConformalOutside} only applies to quasiconformal mappings with 
dilatation (by which we mean $\parallel \mu  \parallel_{\infty}$) at most $ \varepsilon _0$.  If we assume that the dilatation of $ g$ is at most $ \varepsilon _0$, 
then the dilatation of $ f$ satisfies the same bound, so that Lemma~\ref{EstimateForConformalOutside} 
applies to it.

Then by quasi-symmetry, Theorem~\ref{QCConformalInside} and 
Lemma~\ref{EstimateForConformalOutside}, 
\begin{align}\label{EstimateHausdorffContentGESmallDilatation}
\H^{t'}_{\infty} (g E) & \leq   \H^{t'}_{\infty} \bigl(g \bigl(\bigcup_i 12 \cdot P_i\bigr)\bigr) \\
&\leq  
\sum_{i} \diam(g(12 \cdot P_i))^{t'}
\\
&\leq  C(K) \sum_{i} \diam(g( P_i))^{t'} 
\nonumber \\
& \leq  C(K) \Bigl( \sum_{i} \diam(f( P_i))^t \Bigr)^{ \frac{t'}{tK} } 
\\
&\leq 
C(K,t) \Bigl( \sum_{i} \ell ( P_i )^t \Bigr)^{ \frac{t'}{tK} } 
\nonumber \\
& \leq  C(K,t)  \Bigl( \H^t_{\infty} (E) + \varepsilon \Bigr)^{ \frac{t'}{tK} } 
 \leq 
 C(K,t) \varepsilon^{ \frac{t'}{tK} } \; .
\end{align}
The parameter $ \varepsilon >0$ was arbitrary, so the proof of
Lemma~\ref{PropositionAbsoluteContinuityOfHausdorffMeasuresUnderPushForwardByQCMapsSmallDilatationV2} is complete.

\end{proof}

\begin{remark}\label{QuantitativeInequalityHausdorffContents}
The proof of Lemma~\ref{PropositionAbsoluteContinuityOfHausdorffMeasuresUnderPushForwardByQCMapsSmallDilatationV2} actually gives the following quantitative estimate for Hausdorff content. Let $0<t<2$ and $t' = \frac{2Kt}{2+(K-1)t}$. Assume $f$ is a principal $K$-quasiconformal mapping with $\frac{K-1}{K+1} \leq \kappa_0(t)$, and let $E \subset (\frac{1}{32}, \frac{1}{16})^2$ be compact. Then
\begin{equation}\label{HausdorffDistortionInequalityForNormalizedKQCMapsSmallDilatation}
\H^{t'}_{\infty} (f E) \leq C(\kappa_0(t))  \Bigl( \H^t_{\infty} (E)  \Bigr)^{ \frac{t'}{tK} } \;.
\end{equation}
We claim that this can be rewritten in the following invariant form. Assume now $f$ is a $K$-quasiconformal mapping with $\frac{K-1}{K+1} \leq \kappa_0(t)$, and let $E$ be a compact set contained in a ball $B$. Let $t$ and $t'$ be as in \eqref{HausdorffDistortionInequalityForNormalizedKQCMapsSmallDilatation}, and let $\diam A$ denote the diameter of the set $A$. Then
\begin{equation}\label{HausdorffDistortionInequalityForGeneralKQCMapsSmallDilatation}
\frac{\H^{t'}_{\infty} (f E) }{ \left[ \diam fB \right]^{t'} } \leq C(\kappa_0(t))  \left\{ \frac{ \H^t_{\infty} (E) }{ \left[ \diam B \right]^{t}  }  \right\}^{ \frac{t'}{tK} } \;.
\end{equation}
Indeed this follows using the method of Corollary 10 in \cite{astalaiwaniecsaksman} (see also \cite{astalaiwaniecmartin}.) Finally, for arbitrary $K>1$, iteration of \eqref{HausdorffDistortionInequalityForGeneralKQCMapsSmallDilatation} (with $\kappa_0(t')$ instead of $\kappa_0(t)$) shows that \eqref{HausdorffDistortionInequalityForGeneralKQCMapsSmallDilatation} holds, with $C(\kappa_0(t))$ replaced by $C(K,t)$.


\end{remark}

\vskip.5in


\vskip.5in

\end{document}